\documentclass[a4paper,leqno,10pt]{article}

\usepackage{amsmath}
\usepackage{amsthm}
\usepackage{amssymb}
\usepackage{hyperref}
\usepackage{caption}
\usepackage{subcaption}
\usepackage{graphicx}
\usepackage{lineno}

\newtheorem{theorem}{Theorem}
\newtheorem*{maintheorem}{Main Theorem}
\newtheorem{lemma}[theorem]{Lemma}
\newtheorem{proposition}[theorem]{Proposition}

\newtheorem{algorithm}[theorem]{Algorithm}

\newtheorem{remark}[theorem]{Remark}
\numberwithin{equation}{section}
\numberwithin{theorem}{section}

\newcommand{\abs}[1]{|#1|}
\newcommand{\set}[1]{\left\{#1\right\}}

\newcommand{\arrowse}[0]{\overset{e}{\rightarrow}}
\newcommand{\arrowsv}[0]{\overset{v}{\rightarrow}}

\newcommand{\mH}[0]{\mathcal{H}}
\newcommand{\mL}[0]{\mathcal{L}}

\DeclareMathOperator{\G}{G}
\DeclareMathOperator{\V}{V}
\DeclareMathOperator{\E}{E}

\begin{document}


\title{The edge Folkman number\\ $F_e(3, 3; 4)$ is greater than 19}

\author{
	Aleksandar Bikov\thanks{Corresponding author} \hfill Nedyalko Nenov\\
	\let\thefootnote\relax\footnote{Email addresses: \texttt{asbikov@fmi.uni-sofia.bg}, \texttt{nenov@fmi.uni-sofia.bg}}\\
	Faculty of Mathematics and Informatics\\
	Sofia University "St. Kliment Ohridski"\\
	5, James Bourchier Blvd.\\
	1164 Sofia, Bulgaria
}

\maketitle

\begin{abstract}
The set of the graphs which do not contain the complete graph on $q$ vertices $K_q$ and have the property that in every coloring of their edges in two colors there exist a monochromatic triangle is denoted by $\mH_e(3, 3; q)$. The edge Folkman numbers $F_e(3, 3; q) = \min\set{\abs{\V(G)} : G \in \mH_e(3, 3; q)}$ are considered. Folkman proved in 1970 that $F_e(3, 3; q)$ exists if and only if $q \geq 4$. From the Ramsey number $R(3, 3) = 6$ it becomes clear that $F_e(3, 3; q) = 6$ if $q \geq 7$. It is also known that $F_e(3, 3; 6) = 8$ and $F_e(3, 3; 5) = 15$. The upper bounds on the number $F_e(3, 3; 4)$ which follow from the construction of Folkman and from the constructions of some other authors are not good. In 1975 Erdo\"s posed the problem to prove the inequality $F_e(3, 3; 4) < 10^{10}$. This Erdo\"s problem was solved by Spencer in 1978. The last upper bound on $F_e(3, 3; 4)$ was obtained in 2012 by Lange, Radziszowski and Xu, who proved that $F_e(3, 3; 4) \leq 786$. The best lower bound on this number is 19 and was obtained 10 years ago by Radziszowski and Xu. In this paper, we improve this result by proving $F_e(3, 3; 4) \geq 20$. At the end of the paper, we improve the known bounds on the vertex Folkman number  $F_v(2, 3, 3; 4)$ by proving $20 \leq F_v(2, 3, 3; 4) \leq 24$.

\end{abstract}



\section{Introduction}

In this paper only finite, non-oriented graphs without loops and multiple edges are considered. The vertex set and the edge set of a graph $G$ are denoted by $\V(G)$ and $\E(G)$ respectively. Let $G_1$ and $G_2$ be two graphs. Then, we denote by $G_1+G_2$ the graph $G$ for which $\V(G) = \V(G_1) \cup \V(G_2)$ and $\E(G) = \E(G_1) \cup \E(G_2) \cup E'$, where $E' = \set{[x, y] : x \in \V(G_1), y \in \V(G_2)}$, i.e. $G$ is obtained by making every vertex of $G_1$ adjacent to every vertex of $G_2$. The Ramsey number is denoted by $R(p, q)$. More detailed information on the Ramsey numbers, which we use in this paper, can be found in \cite{Rad14}. All undefined terms and notations can be found \cite{Wes01}.\\

The expression $G \arrowse (3, 3)$ means that in every coloring of the edges of the graph $G$ in two colors there is a monochromatic triangle. It is well known that $K_6 \arrowse (3, 3)$. Therefore, if $\omega(G) \geq 6$, then $G \arrowse (3, 3)$. In \cite{EH67} Erd\"os and Hajnal posed the following problem: \emph{Is there a graph $G \arrowse (3, 3)$ with $\omega(G) < 6$}? A positive answer was given by several authors, and Folkman \cite{Fol70} was the first to construct such a graph with clique number 3.

Later, we will use the following fact:
\begin{equation}
\label{equation: G arrowse (3, 3) Rightarrow chi(G) geq 6}
G \arrowse (3, 3) \Rightarrow \chi(G) \geq 6,
\end{equation}
which is a special case of a more general result from \cite{Lin72}.

Denote:

$\mH_e(3, 3; q) = \set{G : G \arrowse (3, 3) \mbox{ and } \omega(G) < q}$.

$\mH_e(3, 3; q; n) = \set{G : G \in \mH_e(3, 3; q) \mbox{ and } \abs{\V(G)} = n}$.

The edge Folkman number $F_e(3, 3; q)$ is defined with

$F_e(3, 3; q) = \min{\set{\abs{\V(G)} : G \in \mH_e(3, 3; q)}}$

According to Folkman`s result \cite{Fol70}, $F_e(3, 3; q)$ exists if and only if $q \geq 4$.

Since $K_6 \arrowse (3, 3)$ and $K_5 \not\arrowse (3, 3)$, $F_e(3, 3; q) = 6, q \geq 7$.

In 1968 Graham \cite{Gra68} obtained the equality $F_e(3, 3; 6) = 8$ by proving that $K_3+C_5 \arrowse (3, 3)$.

Nenov \cite{Nen81} constructed a 15-vertex graph in $\mH_e(3, 3; 5)$ in 1981, thus proving $F_e(3, 3; 5) \leq 15$. In 1999 Piwakowski, Radziszowski and Urba{\'n}ski \cite{PRU99} completed the computation of the number $F_e(3, 3; 5)$ by proving with the help of a computer that $F_e(3, 3; 5) \geq 15$. They also obtained all graphs in $\mH_e(3, 3; 5; 15)$.

The graph $G$ obtained by the construction of Folkman, for which $G \arrowse (3, 3)$ and $\omega(G) = 3$, has a very large number of vertices. Because of this, in 1975 Erd\"os \cite{Erd75} posed the problem to prove the inequality $F_e(3, 3; 4) < 10^{10}$. In 1986 Frankl and R\"odl \cite{FR86} almost solved this problem by showing that $F_e(3, 3; 4) < 7.02 \times 10^{11}$. In 1988 Spencer \cite{Spe88} proved the inequality $F_e(3, 3; 4) < 3 . 10^9$, after an erratum by Hovey, by using probabilistic methods. In 2008 Lu \cite{Lu08} constructed a 9697-vertex graph in $\mH_e(3, 3; 4)$, thus considerably improving the upper bound on $F_e(3, 3; 4)$. Soon after that, Lu`s result was improved by Dudek and R\"odl \cite{DR08}, who proved $F_e(3, 3; 4) \leq 941$. The best upper bound known on $F_e(3, 3; 4)$ was obtained in 2012 by Lange, Radziszowski and Xu \cite{LRX14}, who constructed a 786-vertex graph in $\mH_e(3, 3; 4)$, thus showing that $F_e(3, 3; 4) \leq 786$.

In 1972 Lin \cite{Lin72} proved that $F_e(3, 3; 4) \geq 11$. The lower bound was improved by Nenov \cite{Nen83}, who showed in 1981 that $F_e(3, 3; 4) \geq 13$. In 1984 Nenov \cite{Nen84} proved that every 5-chromatic $K_4$-free graph has at least 11 vertices, from which it is easy to derive that $F_e(3, 3; 4) \geq 14$. From $F_e(3, 3; 5) = 15$ \cite{PRU99} it follows easily, that $F_e(3, 3; 4) \geq 16$. The best lower bound known on $F_e(3, 3; 4)$ was obtained in 2007 by Radziszowski and Xu \cite{RX07}, who proved with the help of a computer that $F_e(3, 3; 4) \geq 19$. According to Radziszowski and Xu \cite{RX07}, any method to improve the bound $F_e(3, 3; 4) \geq 19$ would likely be of significant interest. 

A more detailed view on results related to the numbers $F_e(3, 3; q)$ is given in the book \cite{Soi08}, and also in the papers \cite{RX07}, \cite{Gra12}, \cite{LRX14} and \cite{RX14}.\\

In this paper, we improve the lower bound on the number$F_e(3, 3; 4)$ by proving with the help of a computer the following
\begin{maintheorem}
\label{maintheorem: F_e(3, 3; 4) geq 20}
$F_e(3, 3; 4) \geq 20.$
\end{maintheorem}

At the end of the paper, we consider the vertex Folkman number $F_v(2, 3, 3; 4)$ and prove $20 \leq F_v(2, 3, 3; 4) \leq 24$ (Theorem \ref{theorem: 20 leq F_v(2, 3, 3; 4) leq 24}). This is an improvement over the known bounds $19 \leq F_v(2, 3, 3; 4) \leq 30$ from \cite{SXL09} and \cite{SLHX11}.\\

This paper is organized in 4 sections. In the first section, the necessary notations are introduced, an overview of the related results is given, and the Main Theorem is formulated. In the second section, we present a computer algorithm (Algorithm \ref{algorithm: finding mL_(max)(n; p; s)}), with the help of which in the third section we prove the Main Theorem. In the last fourth section, new bounds on the number $F_v(2, 3, 3; 4)$ are obtained (Theorem \ref{theorem: 20 leq F_v(2, 3, 3; 4) leq 24}).

\section{Algorithm}

Let $G \in \mH_e(3, 3; 4; n)$, $A \subseteq \V(G)$ be an independent set of vertices of $G$, $\abs{A} = s$ and $H = G - A$. Then, obviously, $G$ is a subgraph of $\overline{K}_s + H$, and therefore $\overline{K}_s + H \in \mH_e(3, 3; 5; n)$, from which it is easy to see that $K_1 + H \in \mH_e(3, 3; 5; n - s + 1)$. By this reasoning, in \cite{RX07} it is proved, but not explicitly formulated (see the proofs of Theorem 2 and Theorem 3), the following
\begin{proposition}
\label{proposition: K_1 + H in mH_e(3, 3; 5; n - abs(A) + 1)}
Let $G \in \mH_e(3, 3; 4; n)$, $A \subseteq \V(G)$ be an independent set of vertices of $G$ and $H = G - A$. Then, $K_1 + H \in \mH_e(3, 3; 5; n - \abs{A} + 1)$.
\end{proposition}

With the help of Proposition \ref{proposition: K_1 + H in mH_e(3, 3; 5; n - abs(A) + 1)} in \cite{RX07} the authors prove that every graph in $\mH_e(3, 3; 4; 18)$ can be obtained by adding 4 independent vertices to a 14-vertex graph $H$ such that $K_1 + H \in \mH_e(3, 3; 5; 15)$. All 659 graphs in $\mH_e(3, 3; 5; 15)$ were found in \cite{PRU99}. With the help of a computer it was proved that, by extending such 14-vertex graphs $H$ with 4 independent vertices, it is not possible to obtain a graph in $\mH_e(3, 3; 4; 18)$, i.e. $\mH_e(3, 3; 4; 18) = \emptyset$ and $F_e(3, 3; 4) \geq 19$.

This method is not suitable for proving the Main Theorem, because not all graphs in $\mH_e(3, 3; 5; 16)$ are known and their number is too large. Because of this, we first prove that if $\G \in \mH_e(3, 3; 4; 19)$, then $G$ can be obtained by adding 4 independent vertices to some of 1 139 033 appropriately selected 15-vertex graphs, or by adding 5 independent vertices to some of 113 14-vertex graphs known from \cite{PRU99}. With the help of a computer we check that these extensions do not lead to the construction of a graph in $\mH_e(3, 3; 4; 19)$ and we derive that $\mH_e(3, 3; 4; 19) = \emptyset$ and $F_e(3, 3; 4) \geq 20$. The computer algorithm which we use (Algorithm \ref{algorithm: finding mL_(max)(n; p; s)}), just as the algorithm used in \cite{RX07} to prove that $F_e(3, 3; 4) \geq 19$, is based on some ideas of Algorithm $A_1$ from \cite{PRU99}.

For convenience, we will use the following notations:

$\mL(n; p) = \set{G : \abs{\V(G)} = n, \omega(G) < 4 \mbox{ and } K_p + G \arrowse (3, 3)}$

$\mL(n; p; s) = \set{G \in \mL(n; p) : \alpha(G) = s}$

From $R(3, 4) = 9$ it follows that
\begin{equation}
\label{equation: mL(n; p; s) = emptyset, if n geq 9 and s leq 2}
\mL(n; p; s) = \emptyset, \mbox{ if $n \geq 9$ and $s \leq 2$, }
\end{equation}

and from $R(4, 4) = 18$ it follows that
\begin{equation}
\label{equation: mL(n; p; s) = emptyset, if n geq 18 and s leq 3}
\mL(n; p; s) = \emptyset, \mbox{ if $n \geq 18$ and $s \leq 3$. }
\end{equation}

In \cite{PRU99} it is proved that $\mL(n; 1) \neq \emptyset$ if and only if $n \geq 14$, and all 153 graphs in $\mL(14; 1)$ are found. Later, we will use the following fact:
\begin{equation}
\label{equation: mL(14; 1; s) neq emptyset Leftrightarrow s in set(4, 5, 6, 7)}
\mL(14; 1; s) \neq \emptyset \Leftrightarrow s \in \set{4, 5, 6, 7}, \cite{PRU99}.
\end{equation}

From (\ref{equation: G arrowse (3, 3) Rightarrow chi(G) geq 6}) it follows
\begin{equation}
\label{equation: G in mL(n; p) Rightarrow chi(G) geq 6 - p}
G \in \mL(n; p) \Rightarrow \chi(G) \geq 6 - p.
\end{equation}

Obviously, $\mH_e(3, 3; 4; n) = \mL(n; 0)$. Let $G \in \mH_e(3, 3; 4)$. From the equality $F_e(3, 3; 5) = 15$ \cite{PRU99} and Proposition \ref{proposition: K_1 + H in mH_e(3, 3; 5; n - abs(A) + 1)} it follows that either $\abs{\V(G)} \geq 20$ or $\alpha(G) \leq 5$, and from $R(4, 4) = 18$ it follows that $\alpha(G) \geq 4$. Thus, we obtain
\begin{equation}
\label{equation: mH(3, 3; 4; 19) = mL(19; 0) = mL(19; 0; 4) cup mL(19; 0; 5)}
\mH_e(3, 3; 4; 19) = \mL(19; 0) = \mL(19; 0; 4) \cup \mL(19; 0; 5).
\end{equation}

We will need the following proposition, which follows easily from Proposition \ref{proposition: K_1 + H in mH_e(3, 3; 5; n - abs(A) + 1)}:
\begin{proposition}
\label{proposition: H in mL(n - abs(A), p + 1)}
Let $G \in \mL(n; p)$, $A \subseteq \V(G)$ be an independent set of vertices of $G$ and $H = G - A$. Then $H \in \mL(n - \abs{A}; p + 1)$.
\end{proposition}

We denote by $\mL_{max}(n; p; s)$ the set of all maximal $K_4$-free graphs in $\mL(n; p; s)$, i.e. the graphs $G \in \mL(n; p; s)$ for which $\omega(G + e) = 4$ for every $e \in \E(\overline{G})$. Since every graph in $\mL(19; 0)$ is contained in a maximal $K_4$-free graph, according to (\ref{equation: mH(3, 3; 4; 19) = mL(19; 0) = mL(19; 0; 4) cup mL(19; 0; 5)}) to prove the Main Theorem it is enough to prove that $\mL_{max}(19; 0; 4) = \emptyset$ and $\mL_{max}(19; 0; 5) = \emptyset$. In the proofs of these inequalities we will use Algorithm \ref{algorithm: finding mL_(max)(n; p; s)}, formulated below. 

The graph $G$ is called a $(+K_3)$-graph if $G + e$ contains a new $3$-clique for every $e \in \E(\overline{G})$. We denote by $\mL_{+K_3}(n; p; s)$ the set of all $(+K_3)$-graphs in $\mL(n; p; s)$.

Let $G \in \mL_{max}(n; p; s)$. Let $A \subseteq \V(G)$ be an independent set of vertices of $G$, $\abs{A} = s$ and $H = G - A$. According to Proposition \ref{proposition: H in mL(n - abs(A), p + 1)}, $H \in \mL(n - s, p + 1)$. From $\omega(G + e) = 4, \forall e \in \E(\overline{G})$ it follows that $H$ is $(+K_3)$-graph, and from $\alpha(G) = s$ it follows that $\alpha(H) \leq s$. Therefore, $H \in \mL_{+K_3}(n - s; p + 1; s')$ for some $s' \leq s$. Thus, we proved the following
\begin{proposition}
\label{proposition: H in bigcup(s' leq s)mL_(+K_3)(n - s; p + 1; s')}
Let $G \in \mL_{max}(n; p; s)$. Let $A \subseteq \V(G)$ be an independent set of vertices of $G$, $\abs{A} = s$ and $H = G - A$. Then,

$H \in \bigcup_{s' \leq s}\mL_{+K_3}(n - s; p + 1; s')$.
\end{proposition}

The graph $G$ is called a Sperner graph if $N_G(u) \subseteq N_G(v)$ for some pair of vertices $u, v \in \V(G)$. If $G \in \mL(n; p; s)$ and $N_G(u) \subseteq N_G(v)$, then $G - u \in \mL(n - 1; p; s')$, $s - 1 \leq s' \leq s$. Therefore, every Sperner graph $G \in \mL(n; p; s)$ is obtained by adding one vertex to some graph $H \in \mL(n - 1; p; s')$, $s - 1 \leq s' \leq s$. In the special case, when $G$ is a Sperner graph and $G \in \mL_{max}(n; p; s)$, from $N_G(u) \subseteq N_G(v)$ it follows that $N_G(u) = N_G(v)$, and therefore $G - u \in \mL_{max}(n - 1; p; s')$, $s - 1 \leq s' \leq s$, i.e. $G$ is obtained by duplicating a vertex in some graph $H \in \mL_{max}(n - 1; p; s')$. All non-Sperner graphs in $\mL_{max}(n; p; s)$ are obtained with the help of the following algorithm, which is based on Proposition \ref{proposition: H in bigcup(s' leq s)mL_(+K_3)(n - s; p + 1; s')}:

\begin{algorithm}
\label{algorithm: finding mL_(max)(n; p; s)}
Finding all non-Sperner graphs in $\mL_{max}(n; p; s)$ for fixed $n$, $p$ and $s$.
	
\emph{1.} Find the set of graphs $\mathcal{A} = \bigcup_{s' \leq s}\mL_{+K_3}(n - s; p + 1; s')$. The obtained graphs in $\mL_{max}(n; p; s)$ will be output in $\mathcal{B}$, let $\mathcal{B} = \emptyset$.

\emph{2.} For each graph $H \in \mathcal{A}$:

\emph{2.1.} Find the family $\mathcal{M}(H) = \set{M_1, ..., M_t}$ of all maximal $K_3$-free subsets of $\V(H)$.

\emph{2.2.} Find all $s$-element subsets $N = \set{M_{i_1}, M_{i_2}, ..., M_{i_s}}$ of $\mathcal{M}(H)$, which fulfill the conditions:

(a) $M_{i_j} \neq N_H(v)$ for every $v \in \V(H)$ and for every $M_{i_j} \in N$.

(b) $K_2 \subseteq H[M_{i_j} \cap M_{i_k}]$ for every $M_{i_j}, M_{i_k} \in N$.

(c) If $N' \subseteq N$, then $\alpha(H - \bigcup_{M_{i_j} \in N'} M_{i_j}) \leq s - \abs{N'}$.

\emph{2.3.} For each $s$-element subset $N = \set{M_{i_1}, M_{i_2}, ..., M_{i_s}}$ of $\mathcal{M}(H)$ found in step 2.2 construct the graph $G = G(N)$ by adding new independent vertices $v_1, v_2, ..., v_s$ to $\V(H)$ such that $N_G(v_j) = M_{i_j}, j = 1, ..., s$. If $G$ is not a Sperner graph and $\omega(G + e) = 4, \forall e \in \E(\overline{G})$, then add $G$ to $\mathcal{B}$.

\emph{3.} Remove the isomorph copies of graphs from $\mathcal{B}$.

\emph{4.} Remove from $\mathcal{B}$ all graphs with chromatic number less than $6 - p$.
	
\emph{5.} Remove from $\mathcal{B}$ all graphs $G$ for which $K_p + G \not\arrowse (3, 3)$.
\end{algorithm}

To avoid repetition in the proofs of Theorem \ref{theorem: finding mL_(max)(n; p; s)} and Theorem \ref{theorem: 20 leq F_v(2, 3, 3; 4) leq 24} we formulate the following
\begin{lemma}
\label{lemma: finding max supgraphs}
After the execution of step 2.3 of Algorithm \ref{algorithm: finding mL_(max)(n; p; s)}, the obtained set $\mathcal{B}$ coincides with the set of all maximal $K_4$-free non-Sperner graphs with independence number $s$ which have an independent set of vertices $A \subseteq \V(G), \abs{A} = s$ such that $G - A \in \mathcal{A}$.
\end{lemma}

\begin{proof}
Suppose that in step 2.3 of Algorithm \ref{algorithm: finding mL_(max)(n; p; s)} the graph $G$ is added to $\mathcal{B}$. Let $G = G(N)$ where $N$ and the following notations are the same as in step 2.3. By $H \in \mathcal{A}$, we have $\omega(H) < 4$. Since $N_G(v_j), j = 1, ..., s$, are $K_3$-free sets, it follows that $\omega(G) < 4$. By $H \in \mathcal{A}$, we have $\alpha(H) \leq s$. From this fact and the condition (c) in step 2.2 it follows that $\alpha(G) \leq s$. Since $\set{v_1, ..., v_s}$ is an independent set of vertices of $G$, we have $\alpha(G) = s$ and $G - \set{v_1, ..., v_s} = H \in \mathcal{A}$. The two checks at the end of step 2.3 guarantee that $G$ is a maximal $K_4$-free non-Sperner graph. 

Let $G$ be a maximal $K_4$-free non-Sperner graph with independence number $s$ and $A = \set{v_1, ..., v_s}$ is an independent set of vertices of $G$ such that $H = G - A \in \mathcal{A}$. We will prove that, after the execution of step 2.3 of Algorithm \ref{algorithm: finding mL_(max)(n; p; s)}, $G \in \mathcal{B}$. Since $G$ is a maximal $K_4$-free graph, $N_G(v_i), i = 1, ..., s$, are maximal $K_3$-free subsets of $V(H)$, and therefore $N_G(v_i) \in \mathcal{M}(H), i = 1, ..., s$ (see step 2.1). Let $N = \set{N_G(v_1), ..., N_G(v_s)}$. Since $G$ is not a Sperner graph, $N$ is an $s$-element subset of $\mathcal{M}(H)$ and $N$ fulfills the condition (a) in step 2.2. From the maximality of $G$ it follows that $N$ fulfills the condition (b), and from $\alpha(G) = s$ it follows that $N$ also fulfills (c). Thus, we showed that $N$ fulfills all conditions in step 2.2, and since $G = G(N)$ is a maximal $K_4$-free non-Sperner graph, in step 2.3 $G$ is added to $\mathcal{B}$.
\end{proof}

\begin{theorem}
\label{theorem: finding mL_(max)(n; p; s)}
After the execution of Algorithm \ref{algorithm: finding mL_(max)(n; p; s)}, the obtained set $\mathcal{B}$ coincides with the set of all non-Sperner graphs in $\mL_{max}(n; p; s)$.
\end{theorem}

\begin{proof}
Suppose that, after the execution of Algorithm \ref{algorithm: finding mL_(max)(n; p; s)}, $G \in \mathcal{B}$. According to Lemma \ref{lemma: finding max supgraphs}, $G$ is a maximal $K_4$-free non-Sperner graph and $\alpha(G) = s$. Now, from step 5 it follows that $G \in \mL_{max}(n; p; s)$. 

Conversely, let $G$ be an arbitrary non-Sperner graph in $\mL_{max}(n; p; s)$. Let $A \subseteq \V(G)$ be an independent set of vertices of $G$, $\abs{A} = s$ and $H = G - A$. According to \ref{proposition: H in bigcup(s' leq s)mL_(+K_3)(n - s; p + 1; s')}, $H \in \mathcal{A}$. Now, from Lemma \ref{lemma: finding max supgraphs} we obtain that, after the execution of step 2.3, the graph $G$ is included in the set $\mathcal{B}$. By (\ref{equation: G in mL(n; p) Rightarrow chi(G) geq 6 - p}), after the execution of step 4, $G$ remains in $\mathcal{B}$. It is clear that after step 5, $G$ also remains in $\mathcal{B}$.
\end{proof}

\begin{remark}
\label{remark: delta(G) geq 8}
In \cite{Bik16} it is proved that if $G \in \mH_e(3, 3; 4)$ and $G - e \not\arrowse (3, 3), \forall e \in \E(G)$, then $\delta(G) \geq 8$. Since $\mL(18; 0) = \emptyset$, it follows easily that for each graph $G \in \mL(19, 0)$ we have $\delta(G) \geq 8$. Using this result we can improve Algorithm \ref{algorithm: finding mL_(max)(n; p; s)} in the case $n = 19, p = 0$ in the following way:

1. In step 1 we remove from the set $\mathcal{A}$ the graphs with minimum degree less than 8 - s.

2. In step 2.2 we add the following conditions for the subset $N$:

(d) $\abs{M_{i_j}} \geq 8$ for every $M_{i_j} \in N$.

(e) If $N' \subseteq N$, then $d_H(v) \geq 8 - s + \abs{N'}$ for every $v \not\in \bigcup_{M_{i_j} \in N'} M_{i_j}$.

This way it is guaranteed that in step 2.3 only graphs $G$ for which $\delta(G) \geq 8$ are added to the set $\mathcal{B}$.

Let us note, however, that in the proof of Theorem \ref{theorem: 20 leq F_v(2, 3, 3; 4) leq 24}, Algorithm \ref{algorithm: finding mL_(max)(n; p; s)} must definitely be used without these changes.

\end{remark}

The \emph{nauty} programs \cite{MP13} have an important role in this paper. We use them for fast generation of non-isomorphic graphs and for graph isomorph rejection.

\section{Proof of the Main Theorem}

According to (\ref{equation: mH(3, 3; 4; 19) = mL(19; 0) = mL(19; 0; 4) cup mL(19; 0; 5)}) it is enough to prove that $\mL_{max}(19; 0; 4) = \emptyset$ and $\mL_{max}(19; 0; 5) = \emptyset$.

\textit{1. Proof of $\mL_{max}(19; 0; 4) = \emptyset$.}

We generate all 11-vertex graphs with a computer and among them we find all 362 439 graphs in $\mL_{+K_3}(11; 2; 3)$ and all 7 949 015 graphs in $\mL_{+K_3}(11; 2; 4)$. Let us denote $\mathcal{A}_1 = \mL_{+K_3}(11; 2; 3) \cup \mL_{+K_3}(11; 2; 4)$. By (\ref{equation: mL(n; p; s) = emptyset, if n geq 9 and s leq 2}),

$\mathcal{A}_1 = \bigcup_{s' \leq 4} \mL_{+K_3}(11; 2; s')$.

We execute Algorithm \ref{algorithm: finding mL_(max)(n; p; s)} with $n = 15$, $p = 1$, $s = 4$, having $\mathcal{A} = \mathcal{A}_1$ on the first step. According to \ref{theorem: finding mL_(max)(n; p; s)}, we find all 5750 non-Sperner graphs in $\mL_{max}(15; 1; 4)$. Among the graphs in $\mL(14; 1)$, which are known from \cite{PRU99}, there are 8 maximal $K_4$-free graphs and they all have independence number 4. By adding a new vertex to each of these 8 graphs which duplicates some of their vertices, we obtain all 20 non-isomorphic Sperner graphs in $\mL_{max}(15; 1; 4)$ (see the text before Algorithm \ref{algorithm: finding mL_(max)(n; p; s)}). Thus, all 5770 graphs in $\mL_{max}(15; 1; 4)$ are obtained. All graphs $G$ for which $\omega(G) < 4$ and $\alpha(G) < 4$ are known and can be found in \cite{McK_r}. There are 640 such 15-vertex graphs, among which we find the only 2 graphs in $\mL_{max}(15; 1; 3)$. From (\ref{equation: mL(n; p; s) = emptyset, if n geq 9 and s leq 2}) it follows that every graph from the set $\mathcal{A}_2 = \mL_{+K_3}(15; 1; 3) \cup \mL_{+K_3}(15; 1; 4)$ is a subgraph of some graph in $\mL_{max}(15; 1; 3)$ or in $\mL_{max}(15; 1; 4)$. By removing edges from the graphs in $\mL_{max}(15; 1; 4) \cup \mL_{max}(15; 1; 3)$ we obtain all graphs in $\mathcal{A}_2$ (1 139 028 graphs in $\mL_{+K_3}(15; 1; 4)$ and 5 graphs in $\mL_{+K_3}(15; 1; 3)$). By (\ref{equation: mL(n; p; s) = emptyset, if n geq 9 and s leq 2}),

$\mathcal{A}_2 = \bigcup_{s' \leq 4} \mL_{+K_3}(15; 1; s')$.

We execute Algorithm \ref{algorithm: finding mL_(max)(n; p; s)} with $n = 19$, $p = 0$, $s = 4$, having $\mathcal{A} = \mathcal{A}_2$ on the first step. In step 2.3, 2 551 314 graphs are added to the set $\mathcal{B}$, 2 480 352 of which remain in $\mathcal{B}$ after the isomorph rejection in step 3. After step 4, 2 597 graphs with chromatic number 6 remain in $\mathcal{B}$. In the end, after executing step 5 we obtain $\mathcal{B} = \emptyset$. From $F_e(3, 3; 4) \geq 19$ it follows that in $\mL(19, 0)$ there are no Sperner graphs and by Theorem \ref{theorem: finding mL_(max)(n; p; s)} we obtain $\mL_{max}(19; 0; 4) = \emptyset$.\\

\textit{2. Proof of $\mL_{max}(19; 0; 5) = \emptyset$.}

From the graphs in $\mL(14; 1)$, which are known from \cite{PRU99}, with the help of a computer we find all 85 graphs in $\mL_{+K_3}(14; 1; 4)$ and all 28 graphs in $\mL_{+K_3}(14; 1; 5)$. Let us denote $\mathcal{A}_3 = \mL_{+K_3}(14; 1; 4) \cup \mL_{+K_3}(14; 1; 5)$. By (\ref{equation: mL(14; 1; s) neq emptyset Leftrightarrow s in set(4, 5, 6, 7)}),

$\mathcal{A}_3 = \bigcup_{s' \leq 5} \mL_{+K_3}(14; 1; s')$.

We execute Algorithm \ref{algorithm: finding mL_(max)(n; p; s)} with $n = 19$, $p = 0$, $s = 5$, having $\mathcal{A} = \mathcal{A}_3$ on the first step. In step 2.3, 502 901 graphs are added to the set $\mathcal{B}$, 251 244 of which remain in $\mathcal{B}$ after the isomorph rejection in step 3. After step 4, 31 graphs with chromatic number 6 remain in $\mathcal{B}$. In the end, after executing step 5 we obtain $\mathcal{B} = \emptyset$. Since there are no Sperner graphs in $\mL(19; 0)$, by  Theorem \ref{theorem: finding mL_(max)(n; p; s)} we obtain $\mL_{max}(19; 0; 5) = \emptyset$. \qed\\

All computations were done on a personal computer. Using one processing core, the time needed to execute Algorithm \ref{algorithm: finding mL_(max)(n; p; s)} in the case $n = 19$, $p = 0$, $s = 4$ is about one hour, and in the case $n = 19$, $p = 0$, $s = 5$ it is about half a minute. Note that in the first case among the 2597 6-chromatic graphs obtained after step 4 of Algorithm \ref{algorithm: finding mL_(max)(n; p; s)}, 794 have a minimum degree of 8 or more. In the second case, the total number of 6-chromatic graphs is 31, 11 of which have a minimum degree of 8 or more. Using the improvements of Algorithm \ref{algorithm: finding mL_(max)(n; p; s)} described in Remark \ref{remark: delta(G) geq 8}, the time needed for computations is reduced almost 2 times in the first case and more than 10 times in the second case.

\begin{remark}
\label{remark: mL_(max)(19; 0; 5) = emptyset}
Let us note that the result $\mL_{max}(19; 0; 5) = \emptyset$ can also be obtained with the algorithm from \cite{RX07}, but more computations have to be performed by the computer.
\end{remark}

\section{New bounds on the vertex Folkman number $F_v(2, 3, 3; 4)$}

Let $G$ be a graph and $a_1, ..., a_s$ be positive integers. The expression $G \arrowsv (a_1, ..., a_s)$ means that in every coloring of $\V(G)$ in $s$ colors there exist $i \in \set{1, ..., s}$ such that there is a monochromatic $a_i$-clique of color $i$.

Define:

$\mH_v(a_1, ..., a_s; q) = \set{G : G \arrowsv (a_1, ..., a_s) \mbox{ and } \omega(G) < q}.$

$F_v(a_1, ..., a_s; q) = \min\set{\abs{\V(G)} : G \in \mH_v(a_1, ..., a_s; q)}.$

The numbers $F_v(a_1, ..., a_s; q)$ are called vertex Folkman numbers. The edge Folkman numbers $F_e(a_1, ..., a_s; q)$ are defined similarly.

Folkman \cite{Fol70} proved that $F_v(a_1, ..., a_s; q)$ exists if and only if $q > \max\set{a_1, ..., a_s}$. In this section, we consider the numbers $F_v(3, 3; q)$ and $F_v(2, 3, 3; q)$, which are related to the edge Folkman number $F_e(3, 3; 4)$. According to Folkman`s result, these two numbers exist if and only if $q \geq 4$.

All the numbers $F_v(3, 3; q)$ are known:
\begin{equation*}
F_v(3, 3; q) = \begin{cases}
5, & \emph{if $q \geq 6$ (obvious)}\\
8, & \emph{if $q = 5$}\\
14, & \emph{if $q = 4$, \cite{Nen81} and \cite{PRU99}}\\
\end{cases}
\end{equation*}

About the numbers $F_v(2, 3, 3; q)$ it is known that:
\begin{equation*}
F_v(2, 3, 3; q) = \begin{cases}
6, & \emph{if $q \geq 7$  (obvious)}\\
9, & \emph{if $q = 6$}\\
12, & \emph{if $q = 5$, \cite{Nen01}}\\
\end{cases}
\end{equation*}

The equalities $F_v(3, 3; 5) = 8$ and $F_v(2, 3, 3; 6) = 9$ are a special case of Theorem 3 от \cite{LRU01}.

The number $F_v(2, 3, 3; 4)$ is not known and its computation seems to be quite difficult. The following bounds are known:
\begin{equation}
\label{equation: 19 leq F_v(2, 3, 3; 4) leq 30}
19 \leq F_v(2, 3, 3; 4) \leq 30, \cite{SXL09} \mbox{ and } \cite{SLHX11}.
\end{equation}

Using the results from this paper (the graphs obtained in the proof of the Main Theorem), we improve the bounds in (\ref{equation: 19 leq F_v(2, 3, 3; 4) leq 30}) by proving the following
\begin{theorem}
\label{theorem: 20 leq F_v(2, 3, 3; 4) leq 24}
$20 \leq F_v(2, 3, 3; 4) \leq 24$.
\end{theorem}

The interest in the number $F_v(2, 3, 3; 4)$ is motivated by our conjecture that the inequality $F_v(2, 3, 3; 4) \leq F_e(3, 3; 4)$ is true.

If this inequality holds, then there would be another possible way to prove a lower bound on the number $F_e(3, 3; 4)$ (the number $F_v(2, 3, 3; 4)$ is easier to bound than $F_e(3, 3; 4)$).

Later we will use the following obvious fact:
\begin{equation}
\label{equation: G arrowsv (2, 3, 3) Rightarrow chi(G) geq 6}
G \arrowsv (2, 3, 3) \Rightarrow \chi(G) \geq 6.
\end{equation}

Define:

$\mH_v(a_1, ..., a_s; q; n) = \set{G : G \in \mH_v(a_1, ..., a_s; q) \mbox{ and } \abs{\V(G)} = n}.$

Let $G \arrowsv (2, 3, 3)$ and $A \subseteq \V(G)$ be an independent set of vertices. Then, obviously, $G - A \arrowsv (3, 3)$. Therefore, the following proposition is true:
\begin{proposition}
\label{proposition: H in mH_v(3, 3; 4; n - abs(A))}
Let $G \in \mH_v(2, 3, 3; 4; n)$ and $A \subseteq \V(G)$ be an independent set of vertices of $G$. Then, $G - A \in \mH_v(3, 3; 4; n - \abs{A})$.
\end{proposition}

It is easy to see that if $G \arrowsv (3, 3)$, then $K_1 + G \arrowse (3, 3)$. Thus, we derive
\begin{proposition}
\label{proposition: mH_v(3, 3; 4; n) subseteq mL(n; 1)}
$\mH_v(3, 3; 4; n) \subseteq \mL(n; 1)$.
\end{proposition}

\begin{table}
	\centering
	\resizebox{\textwidth}{!}
	{
		\begin{tabular}{ | l r | l r | l r | l r | l r | l r | l r | }
			\hline
			\multicolumn{2}{|c|}{\hbox to 2.2cm{$|\E(G)|$ \hfill $\#$}}&\multicolumn{2}{|c|}{\hbox to 2.2cm{$\delta(G)$ \hfill $\#$}}&\multicolumn{2}{|c|}{\hbox to 2.2cm{$\Delta(G)$ \hfill $\#$}}& \multicolumn{2}{|c|}{\hbox to 2.2cm{$\alpha(G)$ \hfill $\#$}}&\multicolumn{2}{|c|}{\hbox to 2.2cm{$|Aut(G)|$ \hfill $\#$}}\\
			\hline
			42			&  1		& 0			& 153		& 7			& 65		& 3			& 5			& 1			& 2 052 543	\\
			43			&  4		& 1			& 1 629		& 8			& 675 118	& 4			& 1 300 452	& 2			& 27 729	\\
			44			&  44		& 2			& 10 039	& 9			& 1 159 910	& 5			& 747 383	& 3			& 9			\\
			45			&  334		& 3			& 34 921	& 10		& 165 612	& 6			& 32 618	& 4			& 850		\\
			46			&  2 109	& 4			& 649 579	& 11		& 80 529	& 7			& 766		& 6			& 22		\\
			47			&  9 863	& 5			& 1 038 937	& 			& 			& 8			& 10		& 8			& 55		\\
			48			&  35 812	& 6			& 339 395	& 			& 			& 			& 			& 10		& 2			\\
			49			&  101 468	& 7			& 6 581		& 			& 			& 			& 			& 12		& 11		\\
			50			&  223 881	& 			& 			& 			& 			& 			& 			& 14		& 4			\\
			51			&  378 614	& 			& 			& 			& 			& 			& 			& 16		& 4			\\
			52			&  478 582	& 			& 			& 			& 			& 			& 			& 20		& 1			\\
			53			&  436 693	& 			& 			& 			& 			& 			& 			& 24		& 4			\\
			54			&  273 824	& 			& 			& 			& 			& 			& 			& 			& 			\\
			55			&  110 592	& 			& 			& 			& 			& 			& 			& 			& 			\\
			56			&  26 099	& 			& 			& 			& 			& 			& 			& 			& 			\\
			57			&  3 150	& 			& 			& 			& 			& 			& 			& 			& 			\\
			58			&  160		& 			& 			& 			& 			& 			& 			& 			& 			\\
			59			&  4		& 			& 			& 			& 			& 			& 			& 			& 			\\
			\hline
		\end{tabular}
	}
	\caption{Statistics of the graphs in $\mL(15; 1)$}
	\label{table: mL(15; 1) statistics}
\end{table}

\begin{table}
	\centering
	\resizebox{\textwidth}{!}
	{
		\begin{tabular}{ | l r | l r | l r | l r | l r | l r | l r | }
			\hline
			\multicolumn{2}{|c|}{\hbox to 2.2cm{$|\E(G)|$ \hfill $\#$}}&\multicolumn{2}{|c|}{\hbox to 2.2cm{$\delta(G)$ \hfill $\#$}}&\multicolumn{2}{|c|}{\hbox to 2.2cm{$\Delta(G)$ \hfill $\#$}}& \multicolumn{2}{|c|}{\hbox to 2.2cm{$\alpha(G)$ \hfill $\#$}}&\multicolumn{2}{|c|}{\hbox to 2.2cm{$|Aut(G)|$ \hfill $\#$}}\\
			\hline
			47			&  2		& 4			& 7			& 8			& 20		& 4			& 5			& 1			& 5			\\
			48			&  5		& 5			& 10		& 			& 			& 5			& 15		& 2			& 12		\\
			49			&  7		& 6			& 3			& 			& 			& 			& 			& 4			& 3			\\
			50			&  3		& 			& 			& 			& 			& 			& 			& 			& 			\\
			51			&  1		& 			& 			& 			& 			& 			& 			& 			& 			\\
			52			&  2		& 			& 			& 			& 			& 			& 			& 			& 			\\
			\hline
		\end{tabular}
	}
	\caption{Statistics of the graphs in $\mL(15; 1) \setminus \mH_v(3, 3; 5; 15)$}
	\label{table: mL(15; 1) setminus mH_v(3, 3; 5; 15) statistics}
\end{table}

\begin{figure}
	\centering
	\includegraphics[trim={0 0 0 490},clip,height=150px,width=150px]{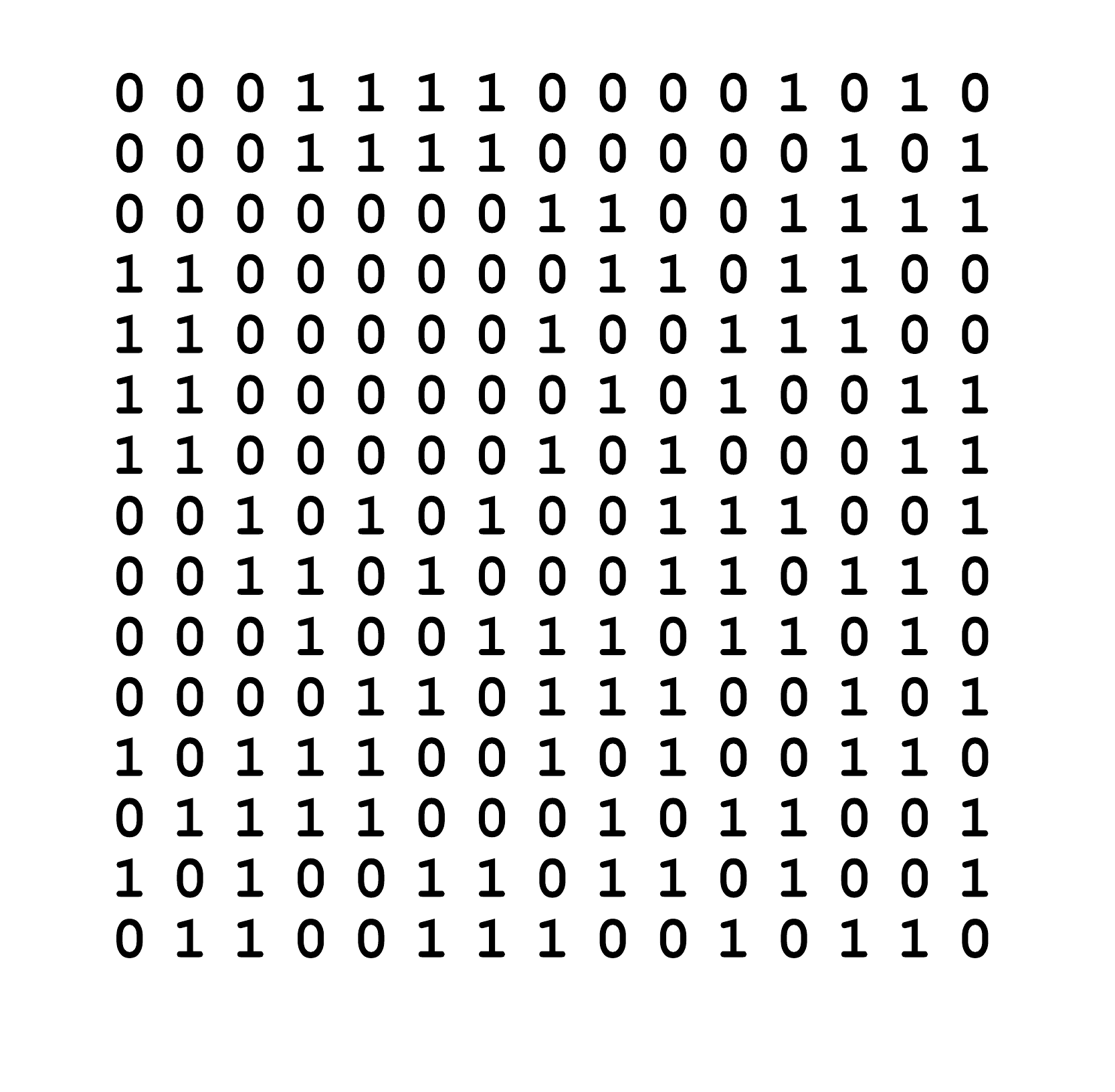}
	\caption{Example of a graph in $\mL(15; 1) \setminus \mH_v(3, 3; 5; 15)$}
	\label{figure: L_15_1_-_fv33_example}
\end{figure}

\begin{remark}
\label{remark: mL(15; 1) and mH_v(3, 3; 4; 15)}
Let us note, that $\mH_v(3, 3; 4; 14) = \mL(14; 1)$ \cite{PRU99}, but $\mH_v(3, 3; 4; 15) \neq \mL(15; 1)$. We found all 2 081 234 graphs $\mL(15; 1)$. In the proof of the Main Theorem we already found the only 2 graphs in $\mL_{max}(15; 1; 3)$ and all 5770 graphs in $\mL_{max}(15; 1; 4)$. With the help of Algorithm \ref{algorithm: finding mL_(max)(n; p; s)}, similarly to the case $s = 4$, we determine $\mL_{max}(15; 1; s), s \geq 5$. We obtain all 826 graphs $\mL_{max}(15; 1; 5)$, all 12 graphs in $\mL_{max}(15; 1; 6)$ and $\mL_{max}(15; 1; s) = \emptyset, s \geq 7$. Thus, we obtain all \mbox{6 610} maximal $K_4$-free graphs in $\mL(15; 1)$. By removing edges from them, we find all 2 081 234 graphs in $\mL(15; 1)$. Some properties of these graphs are listed in Table \ref{table: mL(15; 1) statistics}. Among the graphs in $\mL(15; 1)$ there are exactly 20 graphs, which are not in $\mH_v(3, 3; 4; 15)$. Properties of these 20 graphs are given in Table \ref{table: mL(15; 1) statistics}, and one of these graphs (which has 51 edges) is shown in Figure \ref{figure: L_15_1_-_fv33_example}.
\end{remark}

\subsection*{Proof of theorem \ref{theorem: 20 leq F_v(2, 3, 3; 4) leq 24}}

\textit{1. Proof of the inequality $F_v(2, 3, 3; 4) \geq 20$.}

It is enough to prove that $\mH_v(2, 3, 3; 4; 19) = \emptyset$. Suppose the opposite is true and let $G \in \mH_v(2, 3, 3; 4; 19)$ be a maximal $K_4$-free graph. Since $\mH_v(2, 3, 3; 4; 18) = \emptyset$, $G$ is not a Sperner graph. From Proposition \ref{proposition: H in mH_v(3, 3; 4; n - abs(A))} and $F_v(3, 3; 4) = 14$ it follows that $\alpha(G) \leq 5$, and from $R(4, 4) = 18$ it follows that $\alpha(G) \geq 4$. Therefore, the following two cases are possible:

\emph{Case 1: $\alpha(G) = 4$.} Let $A \subseteq \V(G)$ be an independent set of vertices of $G$, $\abs{A} = 4$ and $H = G - A$. From Proposition \ref{proposition: H in mH_v(3, 3; 4; n - abs(A))} it follows that $H \in \mH_v(3, 3; 4; 15)$ and from Proposition \ref{proposition: mH_v(3, 3; 4; n) subseteq mL(n; 1)} we obtain $H \in \mL(14; 1; s), s \leq 4$. Since $G$ is maximal $K_4$-free graph, it follows that 
\begin{equation}
\label{equation: H in mL_(+K_3)(15; 1; s), s leq 4}
H \in \mL_{+K_3}(15; 1; s), s \leq 4.
\end{equation}
We execute Algorithm \ref{algorithm: finding mL_(max)(n; p; s)} ($n = 19, p = 0, s = 4$). According to (\ref{equation: H in mL_(+K_3)(15; 1; s), s leq 4}), $H \in \mathcal{A}$, and therefore, from Lemma \ref{lemma: finding max supgraphs} we obtain that, after the execution of step 2.3 of the algorithm, $G \in \mathcal{B}$. By (\ref{equation: G arrowsv (2, 3, 3) Rightarrow chi(G) geq 6}), $G \in \mathcal{B}$ after step 4. On the other hand, in the first part of the proof of the Main Theorem we found all 2 597 graphs which are in $\mathcal{B}$ after step 4. None of these graphs belongs to $\mH_v(2, 3, 3; 4)$, which is a contradiction.

\emph{Case 2: $\alpha(G) = 5$.} In the same way as in the proof of Case 1, we see that after the execution of step 4 of Algorithm \ref{algorithm: finding mL_(max)(n; p; s)} $G \in \mathcal{B}$. In the second part of the proof of the Main Theorem we found all 31 graphs in $\mathcal{B}$. None of these graphs belongs to $\mH_v(2, 3, 3; 4)$, which is a contradiction.\\

\begin{figure}
	\centering
	\includegraphics[trim={0 0 0 490},clip,height=240px,width=240px]{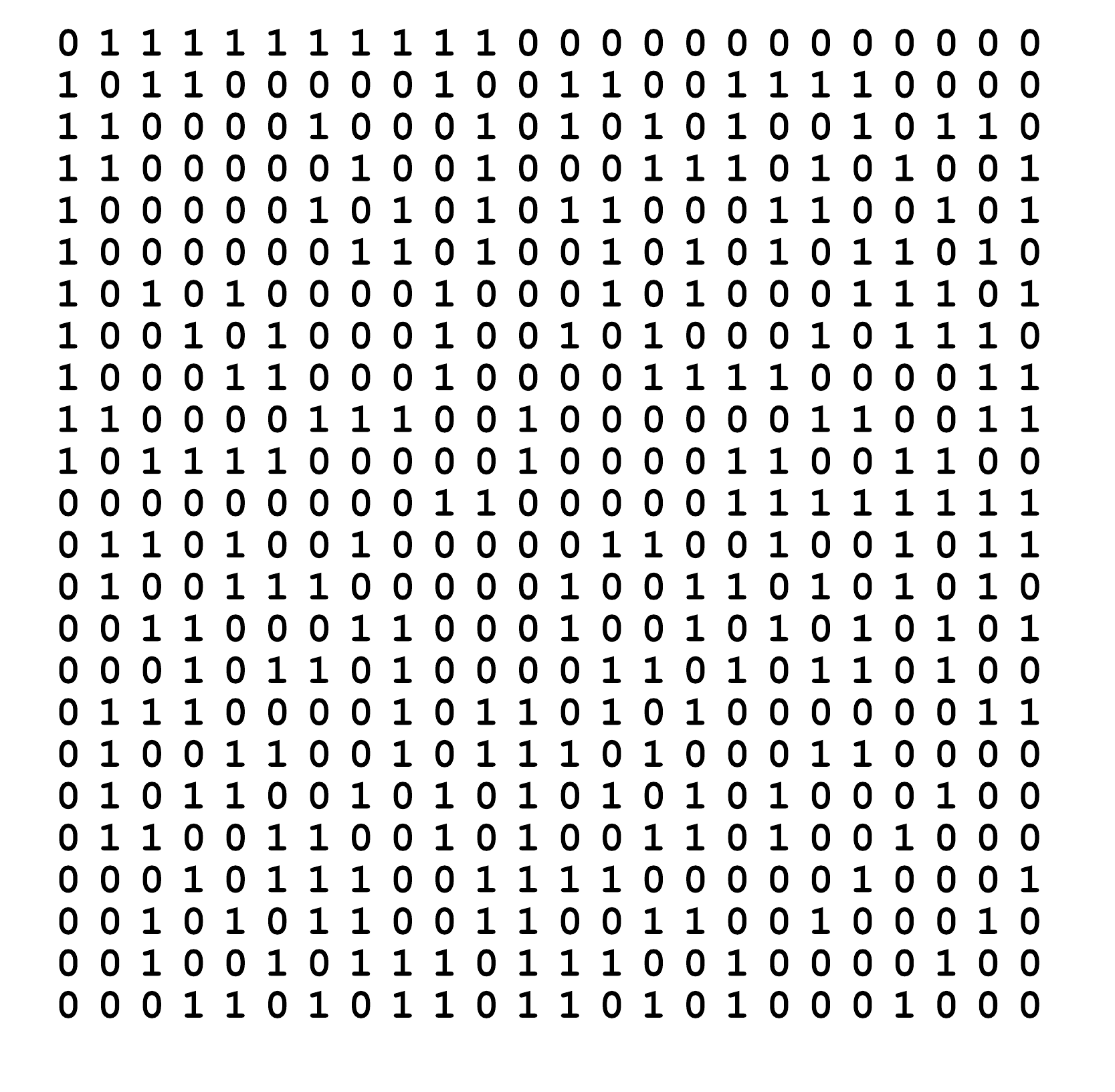}
	\caption{24-vertex transitive graph in $\mH_v(2, 3, 3; 4)$}
	\label{figure: fv233_trans_24}
\end{figure}

\textit{2. Proof of the inequality $F_v(2, 3, 3; 4) \leq 24$.}

All vertex transitive graphs with up to 31 are known and can be found in \cite{Roy_t}. With the help of a computer we check which of these graphs belong to $\mH_v(2, 3, 3; 4)$. In this way, we find one 24-vertex graph, one 28-vertex graph and 6 30-vertex graphs in $\mH_v(2, 3, 3; 4)$. The only 24-vertex transitive graph in $\mH_v(2, 3, 3; 4)$ is given on Figure \ref{figure: fv233_trans_24}. It does not have proper subgraphs in $\mH_v(2, 3, 3; 4)$, but by adding edges to this graph we obtain 18 more graphs in $\mH_v(2, 3, 3; 4; 24)$, two of which are maximal $K_4$-free graphs.





\clearpage


\end{document}